\documentclass[200pt]{amsart}
\usepackage{url, amsmath, amsthm, amssymb,url, amssymb, relsize,placeins, tikz}

\usetikzlibrary{positioning,shapes,shadows}
\usetikzlibrary{arrows,automata}
\newtheorem{theorem}{Theorem}[section]
\newtheorem{remark}[theorem]{Remark}
\newtheorem{lemma}[theorem]{Lemma}

\newtheorem{proposition}[theorem]{Proposition}
\newtheorem{corollary}[theorem]{Corollary}

\newenvironment{defn}[1][]{\refstepcounter{theorem}\begin{trivlist}
\item[\hskip \labelsep {\bfseries Definition  \thetheorem  \, \def\temp{#1}\ifx\temp\empty  #1\else  (#1)\fi
}]}   {\end{trivlist}}

\usepackage{multicol}
\usepackage{hyperref}
\usepackage{multicol}
\usepackage{mathrsfs}
\usepackage{abstract}
\usepackage{graphicx}
\hypersetup{
    colorlinks,
    citecolor=black,
    filecolor=black,
    linkcolor=black,
    urlcolor=black
}

\newcommand{\N}{\mathbb{N}}

\newcommand{\im}{\operatorname{img}}

\newcommand{\C}{\operatorname{Core}}
\DeclareMathOperator{\AdVn}{\A(dV_n)}
\DeclareMathOperator{\OdVn}{\Out(dV_n)}
\DeclareMathOperator{\A}{Aut}
\DeclareMathOperator{\Out}{Out}

\newcommand{\makeset}[2]{\left\lbrace #1 \;\middle|\;
  \begin{tabular}{@{}l@{}}
    #2
   \end{tabular}
  \right\rbrace}

\author{Luke Elliott}
\title{A description of \(\AdVn\) and \(\OdVn\) using transducers}
\begin{document}

\maketitle
\tableofcontents

\begin{abstract}
The groups \(dV_n\) are an infinite family of groups, first introduced by C.
Mart\'inez-P\'erez, F. Matucci and B. E. A. Nucinkis, which includes both the Higman-Thompson groups \(V_n(=1V_n)\) and the Brin-Thompson groups \(nV(=nV_2)\).
A description of the groups \(\operatorname{Aut}(G_{n, r})\) (including the groups \(G_{n,1}=V_n\)) has previously been given by C. Bleak, P. Cameron, Y. Maissel, A. Navas, and F. Olukoya. Their description uses the transducer representations of homeomorphisms of Cantor space introduced a paper of R. I. Grigorchuk,  V. V. Nekrashevich, and V. I. Sushchanskii, together with a theorem of M. Rubin. We generalise the transducers of the latter paper and make use of these transducers to give a description of \(\A(dV_n)\) which extends the description of \(\A(1V_n)\) given in the former paper.
We make use of this description to show that \(\Out(dV_2) \cong \Out(V_2)\mathlarger{\mathlarger{\mathlarger{\wr}}} S_d\),
and more generally give a natural embedding of \(\Out(dV_n)\) into \(\Out(G_{n, n-1}) \mathlarger{\mathlarger{\mathlarger{\wr}}} S_d\).
\end{abstract}

\section{Introduction}

In Matthew G. Brin's 2004 paper \cite{Brin2004}, he introduces the family of simple groups \(dV\), which serve as \(d\)-dimensional analogues to Thompson's group \(V\). The present paper is concerned with finding a ``nice'' way to represent the automorphism groups of the groups \(dV\). To do this we follow a similar path to that in \cite{bleak2016}, but via a more category theoretic perspective. This enables us to prove a conjecture made by Nathan Barker in 2012:

\begin{theorem}\label{1sttheorem}
For all \(d\geq 1\), we have
\(\Out(dV) \cong \Out(V)\mathlarger{\mathlarger{\mathlarger{\wr}}} S_d\) (using the standard action of \(S_d\) on \(d\) points).
\end{theorem}

We view the transducers of \cite{GNS2000} as a category in their own right, and then identify subcategories of transducers which are more appropriate for representing homeomorphisms of ``\(n\)-dimensional" Cantor Spaces. Similarly to \cite{bleak2016}, we then employ Rubin's Theorem \cite{Rubin} to represent the elements of \(\A(dV)\) as homeomorphisms, which in turn are represented with transducers. We also extend the description of \(\Out(V)\) given in \cite{bleak2016} to \(\Out(dV)\).

From this perspective we are able to represent the automorphisms of the encompassing family of groups \(dV_n\), first introduced in the paper \cite{MMN2016cohomological} of Mart\'inez-P\'erez, Matucci, and Nucinkis. We also describe the outer automorphisms of these groups with transducers and give the following theorem extending the one given for \(dV\):

\begin{theorem}\label{2ndtheorem}
For all \(d \geq 1\) and \(n\geq 2\) we have
\[\Out(dV_n) \cong \makeset{\mathbf{T}\in \Out(G_{n, n-1})^d}{\( \prod_{i<d} (\mathbf{T}\boldsymbol{\pi}_i)\overline{\text{sig}} = 1\)}\rtimes  S_d,\]
where the action of \(S_d\) is the standard permutation of coordinates, \(\boldsymbol{\pi}_i\) is the \(i^{th}\) projection map and \(\overline{\text{sig}}\) is the homomorphism of \cite{olukoya2019automorphisms} Definition 7.6.
\end{theorem}
The outer automorphisms of Thompson groups have a history in the literature. In \cite{Brin_AutF, Brin_1998} Brin and Fernando Guzm\'an study the automorphisms of \(F\) and \(T\) type groups. As previously motioned, the authors of \cite{bleak2016} gave a means of describing \(\Out(G_{n, r})\) with transducers, in particular the way the groups \(\Out(1V_n)\) are viewed in this paper is theirs. More recently Feyishayo Olukoya has used transducer based methods to study the outer automorphisms of the groups \(T_{n, r}\) in \cite{OlukoyaAutTnr}. 

The family of groups \(dV\) has also been extensively studied in the literature. In \cite{Brin2004} it is proved that the groups \(dV\) are all infinite, simple, and finitely generated. In \cite{Brin2005} Brin goes on to give an explicit finite presentation for \(2V\) with 8 generators and 70 relations. The paper \cite{Bleak2010} of Collin Bleak and Daniel Lanoue uses Rubin's theorem to show that \(dV\) and \(nV\) are non-isomorphic for \(d\neq n\).

In \cite{hennig2011}, Johanna Hennig and Francesco Matucci show that in general \(dV\) can be finitely presented with \(2d+ 4\) generators and \(10d^2+10d+10\). More recently, Martyn Quick \cite{quick2019} has built much smaller presentations for \(dV\), using only 2 generators as well as \(2d^2 + 3d + 13\) relations.

It is shown in \cite{bleak2016} that \(\A(G_{n, r})\) embeds in the rational group \(\mathcal{R}\) of finite transducers as defined in \cite{GNS2000}.  
In \cite{Belk2016} it is shown that there is a natural topological conjugacy embedding \(2V\) into \(\mathcal{R}\) as well (which can be naturally generalised to \(dV\)). It is therefore natural to ask if this conjugacy sends \(\A(dV)\) to a subgroup of \(\mathcal{R}\). In this paper we give examples to demonstrate that this fails for all \(d\geq 2\) (see Section~\ref{ref}).

Mark V Lawson and Alina Vdovina have constructed many additional Thompson-like groups using the notion of ``\(k\)-monoids" (see \cite{lawson2019higher}). Proposition 3.9 of \cite{lawson2019higher} suggests that the methods of this document are likely only compatible with groups corresponding to the k-monoids which are finite products of finite rank free monoids.

\begin{center}
    
{Acknowledgements}

I would like to thank my supervisor Collin Bleak for reading a draft of this paper, and giving a lot of helpful advice as to how to better present the results within it. I would also like to thank Jim Belk for recommending that I look into the automorphisms of \(2V\).
\end{center}

\section{Preliminaries}
We will compose functions from left to right and we will always index from \(0\). For \(n\in \N\), we will use the notations \(X_n\) and \(\overline{n}\) to denote the set \(\{0, 1, \ldots, n-1\}\). We use the former when thinking of this set as an alphabet, and the latter when thinking of an initial segment of \(\N\). 

We denote the free monoid of all finite words over a finite alphabet \(X\) by \(X^*\). That is
\[X^* := \bigcup_{n \in \mathbb{N}}  X^{n}.\]
This notably includes the empty word which we denote by \(\varepsilon\).
We also follow the standard convention of identifying a letter with a word of length 1. If \(w\in X^*\), then we define \(|w|\) to be the length of \(w\) as a word (which is actually the same as it's cardinality).
If \(X\) is an alphabet, then the set \(X^* \cup X^\omega\) is naturally partially ordered by:
\[x\leq y\text{ if and only if }x\text{ is a ``prefix" of }y.\]
We also extend this partial order to the sets \((X^* \cup X^\omega)^d\) via the product order.

\begin{defn}
We will denote all projection morphisms by \(\boldsymbol{\pi}_0, \boldsymbol{\pi}_1, \ldots\) in all categories with products (the specific morphism used will be determined by the context).
\end{defn}

If \(x\in (X_n^*)^d\), \(y\in (X_n^*\cup X_n^\omega)^d\) and \(x\leq y\), then we define \(y-x\) to be the unique \(z\in (X_n^*\cup X_n^\omega)^d\) such that \(xz=y\) (using coordinate-wise concatenation).

\begin{defn}
If \(n\geq 2\), then we define \(\mathfrak{C}_n:= X_n^\omega\) to be the usual Cantor space with the product topology. Moreover if \(w\in (X_n^*)^d\) then we define \[w\mathfrak{C}_n^d:= \makeset{x\in  \mathfrak{C}_n^d}{\( w\leq x\)}.\]
Note that these sets are clopen, and the collection of all such sets is a basis for \(\mathfrak{C}_n^d\). Such basic open sets will be referred to as \textit{cones}.
\end{defn}
\begin{defn}
If \(X\) is a topological space, then we denote the homeomorphism group of \(X\) by \(H(X)\).
\end{defn}
We can now give the definition of the groups \(dV_n\) which will we will use throughout the paper.
\begin{defn}
Suppose that \(F_1, F_2\) are finite subsets of \((X_n^*)^d\), such that
\[\makeset{w\mathfrak{C}_n^d}{\(w\in F_1\)} \quad  \text{ and } \quad \makeset{w\mathfrak{C}_n^d}{\(w\in F_2\)}\]
are partitions of \(\mathfrak{C}_n^d\), and \(\phi: F_1 \to F_2\) is a bijection. We call such sets \(F_1, F_2\) \textit{complete prefix codes} for \(\mathfrak{C}_n^d\). We then define the \textit{prefix exchange map} \(f_\phi:~\mathfrak{C}_n^d~\to~\mathfrak{C}_n^d\) by:

If \(w\in F_1\) and \(x\in w\mathfrak{C}_n^d\), then \[(x)f_{\phi} =
(w\phi)(x - w).\]

Such prefix exchange maps are always homeomorphisms and the set of all such maps under composition forms the group \(dV_n\) (or just \(dV\) if \(n=2\)).
\end{defn}
\begin{remark}\label{anti-chain sizes}
There is a complete prefix code for \(\mathfrak{C}_n^d\) of size \(m\) if and only if \(m\in  (1 + (n-1)\mathbb{N})\).
\end{remark}

\section{Generalizing the transducers of Grigorchuk,  Nekrashevich, and Sushchanskii}
A transducer, as introduced by Grigorchuk,  Nekrashevich, and Sushchanskii (which we shorten to GNS) in \cite{GNS2000}, can be thought of as a way of assigning each letter of an alphabet, a transformation of a ``state set", together with a word to write for each state. These are then extended to all words in the input alphabet via the universal property of the free monoid. With reading elements of a monoid in mind, the following is a natural generalisation of their transducer definition.
\begin{defn}
We say that \(T := (Q_T, D_T, R_T, \pi_T, \lambda_T)\) is a transducer if:
\begin{enumerate}
    \item \(Q_T\) is a set (called the set of states),
    \item \(D_T\) is a semigroup (called the domain semigroup),
    \item \(R_T\) is a semigroup (called the range semigroup),
    \item \(\pi_T: Q_T\times D_T \to Q_T\) is a (right) action of \(D_T\) on the set \(Q_T\) (called the transition function),    \item \(\lambda_T : Q_T\times D_T \to R_T\) is a function with the property that for all \(q\in Q_T\) and \(s,t\in D_T\) we have \[( q,st)\lambda_T= (q,s)\lambda_T(( q,s)\pi_T,t)\lambda_T\text{ (called the output function).}\] 
\end{enumerate}
We will often refer to the domain semigroup and range semigroup of a transducer as simply it's domain and range.
\end{defn}
Note that a one state transducer is equivalent to a semigroup homomorphism.
\begin{defn}
Let \(A, B\) be transducers. We say that \(\phi\) is a transducer homomorphism from \(A\) to \(B\) (written \(\phi:A \to B\)), if \(\phi\) is a 3-tuple \((\phi_Q, \phi_D, \phi_R)\) with the following properties: 
\begin{enumerate}
    \item \(\phi_{R}:R_{A}\to R_{B}\) is a semigroup homomorphism,
    \item \(\phi_{D}:D_{A}\to D_{B}\) is a semigroup homomorphism,
    \item \(\phi_Q:Q_{A} \to Q_B\) is a function, such that for all \(q\in Q_{A}\) and \(s\in D_A\) we have
    \[(q, s)\pi_{A}\phi_Q = (q\phi_Q, s\phi_D)\pi_{B} \quad \text{and}\quad (q, s)\lambda_{A}\phi_R = (q\phi_Q, s\phi_D) \lambda_{B}.\]
\end{enumerate}
 If furthermore, the maps \(\phi_D, \phi_R\) are identity maps, then we say that \(\phi\) is \textit{strong}.
\end{defn}
\begin{remark}
If transducer homomorphisms are composed component-wise, then transducers become a category when together with transducer homomorphisms, or strong transducer homomorphisms.
\end{remark}

\begin{defn}
A transducer homomorphism \(\phi\), is called a \textit{quotient} map if each of \(\phi_{Q}, \phi_D, \phi_R\) is surjective.
\end{defn}

We say that transducers \(A\) and \(B\) are isomorphic (denoted \(A\cong B\)) if they are isomorphic in the category of transducers and transducer homomorphisms. Similarly, we say that \(A\) and \(B\) are strongly isomorphic (denoted \(A\cong_S B\)) if they are isomorphic in the category of transducers and strong transducer homomorphisms.

The next definition gives us a means of minimizing our transducers which coincides with the GNS notion of combining equivalent states.

\begin{defn}
If \(T\) is a transducer with \(R_T\) cancellative, then we define its \textit{minimal transducer} \(M_T\) to be \((Q_T/\sim_{M_T}, D_T, R_T, \pi_{M_T}, \lambda_{M_T})\) where \(\sim_{M_T}, \pi_{M_T}\) and \(\lambda_{M_T}\) are defined by:
\begin{enumerate}
    \item \(\sim_{M_T}\) is the equivalence relation 
    \[\makeset{(p, q)\in Q_T^2}{\((p, s)\lambda_{T} = (q, s)\lambda_{T}\) for all \(s\in D_T\)},\]
    \item if \(q\in Q_T\), \(s\in D_T\) then \(([q]_{\sim_{M_T}}, s)\pi_{M_T} = [(q,s)\pi_T]_{\sim_{M_T}}\),
    \item if \(q\in Q_T\), \(s\in D_T\) then \(([q]_{\sim_{M_T}}, s)\lambda_{M_T} = (q,s)\lambda_T\).
\end{enumerate}
It is routine to verify that this is well-defined (as \(R_T\) is cancellative) and the natural strong quotient candidate \(q_T:T\to M_T\), with \((p){q_T}_Q= [p]_{\sim_{M_T}}\) is a strong quotient map.
\end{defn}
\begin{lemma}
If \(A\) is a transducer with \(R_A\) cancellative, then all strong quotient maps \(\phi:A\to B\) are left divisors of \(q_A\).
\end{lemma}
\begin{proof}
We define \(\psi:B\to M_A\) by having \(\psi_D, \psi_R\) be the identity maps and defining \(\psi_Q\) by:
\[((q)\phi_Q)\psi_Q:= (q){q_A}_Q.\]
If we can show that \(\psi\) is a well defined transducer homomorphism then the result follows.

Note that all the maps \(\phi_D, \phi_R, \psi_D, \psi_R, {q_A}_D\) and \({q_A}_R\) are identity maps so we can ignore them for the purposes of this proof. We first show that \(\psi\) is well-defined.  Suppose that \(q_0, q_1\in Q_A\) satisfy \((q_0)\phi_Q= (q_1)\phi_Q\). We need to show that \(q_0\sim_{M_T} q_1\). Let \(s\in D_A\) be arbitrary, then
\[(q_0, s)\lambda_A = ((q_0)\phi_Q, s)\lambda_B= ((q_1)\phi_Q, s)\lambda_B= (q_1, s)\lambda_A.\]
We next show that \(\psi\) is a homomorphism. Using the fact that \(\phi_Q\) is surjective, let \((p,s)= ((q)\phi_Q,s)\in Q_B~\times~D_B\) be arbitrary. We need only verify that \(\psi\) satisfies condition 3 from the definition of a transducer homomorphism.  We have
\begin{align*}
    (p, s)\pi_B\psi_Q&=((q)\phi_Q, s)\pi_B\psi_Q&\text{ by the definition of }p\\
    &=(q, s)\pi_A\phi_Q\psi_Q& \text{ because }\phi\text{ is a homomorphism} \\
    &=((q){q_A}_Q, s)\pi_{M_A} &\text{ because }q_A\text{ is a homomorphism}\\
    &=(((q)\phi_Q)\psi_Q, s)\pi_{M_A}&\text{ by the definition of }\psi\\
    &=((p)\psi_Q, s)\pi_{M_A}&\text{ by the definition of }p,
\end{align*}
and similarly \((p, s)\lambda_B=((p)\psi_Q, s)\lambda_{M_A}\) as required.

\end{proof}
\begin{defn}
If \(T\) is a transducer, and we restrict \(Q_T, D_T\), and \(R_T\) to sets which are (together) closed under the transition and output functions, then we obtain another transducer. We call such a transducer a \textit{subtransducer} of \(T\).
\end{defn}

\begin{defn}
If \(d\in \N\) and \(n\in \N\backslash \{0, 1\}\), then we define a \((d,n)\)-transducer to be a transducer \(T\) with \((X_n^*)^d\) as its domain and range, and such that the transition function is a monoid action (so you don't transition when reading the identity).
\end{defn}
If \(T\) is a transducer, \(q\in Q_T\), and \(w\in D_T\), then will view the maps \((q, \cdot)\pi_T\), and \((q, \cdot)\lambda_T\) as reading \(w\) though a path in \(T\) from \(q\), ending at the state \((q, w)\pi_T\), and writing \((q, w)\lambda_T\) along the way (similarly to GNS transducers).
Note that unlike GNS transducers, there isn't always a ``best" way of splitting up this path into minimal steps, for example \((0,0)\in (X_2^*)^2\) could naturally be decomposed as either \((0, \varepsilon)(\varepsilon, 0)\) or  \(( \varepsilon, 0)(0, \varepsilon)\).

If \(T\) is a \((d,n)\)-transducer then (like GNS transducers) we can naturally extend this idea to ``infinite words", which in this case means elements of \((X_n^\omega)^d\), by reading arbitrarily long finite prefixes of an element and taking the limit of the elements written.

\begin{defn}
If \(T\) is a \((d, n)\)-transducer and \(q\in Q_T\), then we define \(f_{T,q}:(X_n^\omega)^d \to (X_n^\omega \cup X_n^*)^d\) to be the map which maps \(w\in (X_n^\omega)^d\) to the word written when \(w\) is read in \(T\) from the state \(q\).
\end{defn}
Note that if \(A\) is a \((d, n)\)-transducer, \(q\in Q_A\) and \(\phi:A\to B\) is a strong transducer homomorphism, then \(f_{A,q}=f_{B, (q)\phi_Q}\). In particular this is true of the homomorphism \(q_A\).
\begin{defn}
Similarly to GNS we say that a \((d, n)\)-transducer \(T\) is \textit{degenerate} if there exist \(q\in Q_T\), \(i\in \overline{d}\) and \(x\in \mathfrak{C}_n^d\) such that \((x)f_{T,q}\pi_i\) is finite.
\end{defn}
We will often use the following fact without comment.
\begin{remark}
If \(T\) is a non-degenerate \((d,n)\)-transducer and \(q\in Q_T\), then for all \(m\in \N\) there is \(k\in \N\) such that reading an element of \((X _n^k)^d\) always writes a word whose length is at least \(m\) in every coordinate.
\end{remark}

There are 2 important ways by which we combine our transducers, there is ``composition" as was done in GNS, and taking products in the categorical sense.

\begin{defn}
If \(A\) and \(B\) are transducers, such that the range of \(A\) is contained in the domain of \(B\), then we define their \textit{composite} by
\[AB= (Q_{AB}, D_{AB}, R_{AB}, \pi_{AB}, \lambda_{AB}).\]
Where
\begin{enumerate}
    \item \(Q_{AB}:= Q_A\times Q_B\), \(D_{AB}:= D_A\), \(R_{AB}:= R_B\),
    \item \(((a, b), s)\pi_{AB}= ((a,s)\pi_A, (b, (a,s)\lambda_{A})\pi_B)\),
    \item \(((a, b), s)\lambda_{A,B}= (b,(a,s)\lambda_{A})\lambda_B\).
\end{enumerate}
\end{defn}
As was the case in GNS, this definition is constructed so that whenever \(A, B\) are non-degenerate \((d, n)\)-transducers, and \((p,q)\in A\times B\), we obtain \(f_{A, p}f_{B,q}=f_{AB,(p,q)}\).
\begin{defn}
If \((A)_{i\in I}\) are transducers, then we define
\(\prod_{i\in I}A_i := P\)
where
\[Q_P:=\prod_{i\in I}Q_{A_i}, \quad D_P:= \prod_{i\in I}D_{A_i}, \quad R_P := \prod_{i\in I}R_{A_i},\]
and for all \((p_i)_{i\in I}\in Q_P\) and \((s_i)_{i\in I}\in D_P\) we have
\[((p_i)_{i\in I}, (s_i)_{i\in I})\pi_{P}=
((p_i,s_i)\pi_{A_i})_{i\in I},\]
\[((p_i)_{i\in I}, (s_i)_{i\in I})\lambda_{P}=
((p_i,s_i)\lambda_{A_i})_{i\in I}.\]
For \(i\in I\) we then define \(\boldsymbol{\pi}_i: P \to A_i\) to be the transducer homomorphism
\((\boldsymbol{\pi}_i, \boldsymbol{\pi}_i, \boldsymbol{\pi}_i)\). One can verify that this is a product in the category theoretic sense (using transducer homomorphisms but not strong transducer homomorphisms).
\end{defn}
The following definition gives us, for each homeomorphism \(h\) of \(\mathfrak{C}_n^d\), a transducer \(M_h\) representing it. From the definition, one can see that this transducer has no inaccessible states, has complete response and has no distinct but equivalent states. So in particular when \(d= 1\), the transducer \(M_h\) is the minimal transducer representing \(h\) as described by GNS. 
\begin{defn}
If \(h\in H(\mathfrak{C}_n^d)\), then we define \(T_h\) to be the \((d, n)\)-transducer with
\begin{enumerate}
    \item \(Q_{T_h}:= (X_n^*)^d\),
    \item \((s, t)\pi_{T_h}= st\),
    \item \(((s, t)\lambda_{T_h})\boldsymbol{\pi}_i\) is \(b-a\), where \(b\) is the longest common prefix of the words in the set \(((st\mathfrak{C}_n^d)h)\boldsymbol{\pi}_i\) and \(a\) is the longest common prefix of the words in the set \(((s\mathfrak{C}_n^d)h)\boldsymbol{\pi}_i\).
\end{enumerate}
(As \(h\) is a homeomorphism, the set \((st\mathfrak{C}_n^d)h\) is always open and thus \((s, t)\lambda_{T_h}\) is always an element of \((X_n^*)^d\).)
Moreover, as was the case in GNS, if \(q=1_{(X_n^*)^d}\) then \(f_{T_h, q}= h\). We also define \(M_h:= M_{T_h}\).
\end{defn}
\begin{remark}\label{inj_clo}
If \(h\in H(\mathfrak{C}_n^d)\) and \(q\in Q_{M_h}\), then \(f_{M_h, q}\) is injective with clopen image.
\end{remark}
The proof of the following theorem is analogous to the proof of the analogous theorem in GNS, (the above construction deals with the homeomorphism case).
\begin{theorem}\label{epic}
A function \(h: \mathfrak{C}_n^d \to \mathfrak{C}_n^d\) is continuous if and only if there is a non-degenerate \((d, n)\)-transducer \(T\) and \(q\in Q_T\) such that \(h= f_{T,q}\).
\end{theorem}
\section{Generalizing the synchronizing homeomorphisms of Bleak, Cameron, Maissel, Navas, and Olukoya}
As was the case in \cite{bleak2016} when analyzing \(\A(G_{n, r})\), we now want to restrict to the transducers which give us the automorphisms we want. We thus extend the notion of synchronization given there.
\begin{defn}
We say that a \((d, n)\)-transducer \(T\) is \textit{synchonizing at level }\(k\) if for all \(q_1,q_2\in Q_T\) and \(w\in (X_n^*)^d\) with \(\min(\makeset{|w\boldsymbol{\pi}_i|}{\(i\in \overline{d}\)})\geq k\), we have \((q_1, w)\pi_T = (q_2, w)\pi_T\).
We say that \(T\) is \textit{synchronizing} if there is a level at which it is synchronizing. 
The \textit{synchronizing length} of a synchronizing transducer \(T\) is \[\min(\makeset{k\in \N}{\(T\) is synchronizing at level \(k\)}).\]
In this case we define the function
\[\mathfrak{s}_T:= \makeset{(w, q)\in (X_n^*)^d\times Q_T}{for all \(p\in Q_T\) we have \((w,p)\pi_T=q\)}.\]
So \(\mathfrak{s}_T\) is basically \(\pi_T\) restricted to the part of it's domain where the input state is not needed. The image of \(\mathfrak{s}_T\), denoted \(\C(T)\), is called the \textit{core} of \(T\).
\end{defn}
It is useful to think of the core of a synchronising \((d,n)\)-transducer as the place reached when a sufficient amount of information has been read in each coordinate.
In particular, if a word is read from any core state of a synchronizing \((d, n)\)-transducer \(T\) then you stay in the core, thus \(\C(T)\) is a (synchronizing) subtransducer of \(T\) (when given the restrictions of the transition and output functions of \(T\)).

As \(\C(T)= ((X_n^k)^d)\mathfrak{s}_T\) (where \(k\) is the synchronizing length of \(T\)), it follows that \(\C(T)\) is always finite.

The following proposition is routine to verify, and shows that our transducer framework describes \(dV_n\) in a manner analogous to the way in which the transducers of GNS describe \(V_n\).
\begin{proposition}\label{nV_placement}
If \(h\in H(\mathfrak{C}_n^d)\), then \(h\in dV_n\) if and only if \(M_h\) is a synchronizing transducer whose core consists of a single ``identity" state.
\end{proposition}
Unlike for \(V_n\), the transducers for elements of \(dV_n\) can sometimes be infinite. For example the transducer representing the baker's map of 2V is infinite (as can be seen in Figure~\ref{baker}) as it can't write anything until something is read in the first coordinate.
\begin{figure}
\begin{center}
\begin{tikzpicture}[->,>=stealth',shorten >=1pt,auto,node distance=5cm,on grid,semithick,
                    every state/.style={fill=red,draw=none,circular drop shadow,text=white}]  
  \node [state] (A)                {Start};
  \node [state] (B)  [right= of A] {};
  \node [state] (C)  [below= of B] {Core};
  \node [state] (D)  [below= of A] {};
  \node [] (E) [below= of D] {$\vdots$};
  \node [] (G) [right= of E] {$\ddots$};
  \node [] (F) [right= of B] {$\ldots$};
  \node [] (H) [below= of F] {$\ddots$};
 \path [->]
 (A) edge [out=305,in= 145] node [swap]
 {$(1, \varepsilon)/(\varepsilon, 1)$} (C)
 (A) edge [out=325,in=125] node [swap,xshift = 17pt,yshift = 50pt]
 {$(0, \varepsilon)/(\varepsilon, 0)$} (C)
 (A) edge [out=0,in=180] node [swap, yshift= 17pt]
 {$(\varepsilon,1)/(\varepsilon,\varepsilon)$} (B)
 (A) edge [out=270,in=90] node [swap]
 {$(\varepsilon,0)/(\varepsilon,\varepsilon)$} (D)
 
 (B) edge [out=280,in=85] node [swap, xshift=63pt, yshift = -20pt]
 {$(1, \varepsilon)/(\varepsilon, 11)$} (C)
 (B) edge [out=260, in= 95] node [swap, yshift= 20pt]
 {$(0, \varepsilon)/(\varepsilon, 01)$} (C)
 (B) edge [out=00,in=180] node [swap, yshift=17pt]
 {$(\varepsilon,1)/( \varepsilon,\varepsilon)$} (F)
 (B) edge [out=315,in=135] node []
 {$(\varepsilon,0)/( \varepsilon,\varepsilon)$} (H)
 
 (C) edge  [out=310,in=335, loop]
 node  [swap]{$(1, \varepsilon)/(1, \varepsilon)$} (C)
 (C) edge  [out=350,in=15, loop]
 node [swap]{$(0, \varepsilon)/(0, \varepsilon)$} (C)
 (C) edge [out=220,in=245, loop] node [swap]
 {$(\varepsilon,1)/( \varepsilon,1)$} (C)
 (C) edge [out=270,in=295, loop] node [swap]
 {$(\varepsilon,0)/( \varepsilon,0)$} (C)
 
 (D) edge [out=10,in=170] node [swap, yshift=17pt]
 {$(1, \varepsilon)/(\varepsilon, 10)$} (C)
 (D) edge [out=350,in=190] node [swap]
 {$(0, \varepsilon)/(\varepsilon, 00)$} (C)
 (D) edge [out=315,in=135] node [swap, xshift=70pt]
 {$(\varepsilon,1)/(\varepsilon, \varepsilon)$} (G)
 (D) edge [out=270, in= 90] node [swap]
 {$(\varepsilon,0)/(\varepsilon, \varepsilon)$} (E);
 \end{tikzpicture}
 \end{center}
\caption{Part of a minimal transducer with \((X_2^*)^2\) as domain and range. This represents the baker's map in \(2V\) (this transducer is infinite and every state has an edge with the core as it's target).}\label{baker}
\end{figure}
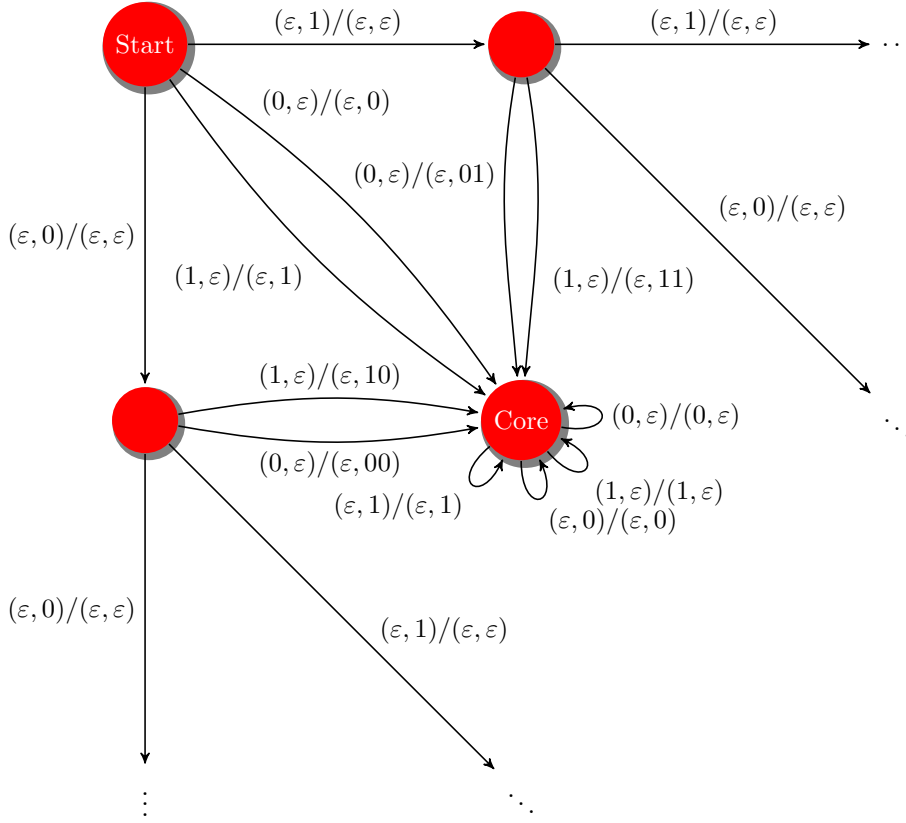

We will now introduce the monoids \(d\mathcal{S}_{n,1}\), \(\widetilde{d\mathcal{O}_{n,1}}\), \(d\mathcal{B}_{n, 1}\) and \(d\mathcal{O}_{n,1}\) which generalise the monoids \(\mathcal{S}_{n,1}\), \(\widetilde{\mathcal{O}_{n,1}}\), \(\mathcal{B}_{n, 1}\) and \(\mathcal{O}_{n,1}\) of \cite{bleak2016}.
\begin{defn}
We say an element \(f\in H(\mathfrak{C}_n^d)\) is \textit{synchronizing} if \(M_f\) is synchronizing. We define \(d\mathcal{S}_{n,1}\) to be the set of synchronizing elements of \(H(\mathfrak{C}_n^d)\).
\end{defn}

\begin{remark}\label{exp_remark}
If \(A, B\) are synchronizing, non-degenerate transducers then so is their composite \(AB\). This works as you can synchronize the first coordinate using the synchronizing property of \(A\), and once the first coordinate is in the finite non-degenerate core of \(A\), one can read enough so that the output of \(A\) synchronizes \(B\) as well.
\end{remark}
\begin{corollary}
If \(f, g\in d\mathcal{S}_{n, 1}\) then 
\(\C(M_fM_g)\) is a subtransducer of the composite transducer \(\C(M_f)\C(M_g).\)
\end{corollary}
\begin{corollary}\label{sync_monoid}
The set \(d\mathcal{S}_{n, 1}\) is always a monoid.
\end{corollary}
\begin{proof}
It follows from Remark~\ref{exp_remark} that, if \(f, g\in d\mathcal{S}_{n, 1}\), then \(M_fM_g\) is synchonizing.

We will now essentially minimise \(M_fM_g\) in the GNS fashion and obtain \(M_{fg}\). Let \(q_f:=(1_{(X_n^*)^d}){q_{T_f}}_Q\) and \(q_g :=(1_{(X_n^*)^d}){q_{T_g}}_Q\). We have
\[fg=f_{M_{f},q_f}f_{M_g, q_g}= f_{M_{f}M_{g}, (q_f, q_g)} .\]

We then define \(A\) to be the transducer with the same states, domain, range and transition function as \(M_fM_g\) but with \((q, w)\lambda_A= b-s\), where \(b\) is the longest common prefix of the set \((w\mathfrak{C}_n^d)f_{M_fM_g, q}\) and \(s\) is the longest common prefix of the set \((\mathfrak{C}_n^d)f_{M_fM_g, q}\) (it follows from Remark~\ref{inj_clo} that \(b\) and \(s\) are finite). Let \(A'\) be the subtransducer of \(A\) consisting of the states that are accessible from \((q_f, q_g)\) (the image of \(((q_f, q_g),\cdot)\pi_A\)).

It follows from the definition that \(A'\) is a strong quotient of the transducer \(T_{fg}\). Thus \(M_{fg}\) is a strong quotient of \(A'\). As \(A\) has the same transitions as \(M_fM_g\), it follows that \(A\) is synchronizing. Moreover, since \(A'\) is a subtransducer of \(A\), we get that \(A'\) is synchronizing and thus \(M_{fg}\) is also synchronizing (as a strong quotient of \(A'\)).
\end{proof}
\begin{corollary}
The set \(\widetilde{d\mathcal{O}_{n,1}}:=\makeset{[\C(M_f)]_{\cong_S}}{\(f\in d\mathcal{S}_{n,1}\)}\) naturally forms a monoid, which is a quotient of \(d\mathcal{S}_{n,1}\).
\end{corollary}
\begin{proof}
 We define
\[[\C(M_f)]_{\cong_S}[\C(M_g)]_{\cong_S}=[\C(M_{fg})]_{\cong_S}.\]
This is well defined as the strong isomorphism type \(\C(M_{fg})\) can be found by removing incomplete response from \(\C(M_f)\C(M_g)\), combining equivalent states and passing to the core (in the same manner as the proof of Corollary~\ref{sync_monoid}).
\end{proof}
\begin{defn}
We define \(d\mathcal{B}_{n, 1}\) to be the group of units of \(d\mathcal{S}_{n, 1}\), and \(d\mathcal{O}_{n,1}:=\makeset{[\C(M_f)]_{\cong_S}}{\(f\in d\mathcal{B}_{n,1}\)}\).
\end{defn}
\begin{lemma}\label{nV_is_normal}
The map \(f\mapsto [\C(M_f)]_{\cong_S}\) is a surjective group homomorphism from \(d\mathcal{B}_{n, 1}\) to \(d\mathcal{O}_{n,1}\) with kernel \(dV_n\).
\end{lemma}
\begin{proof}
This map is a homomorphism by the definition of multiplication in \(\widetilde{d\mathcal{O}_{n,1}}\), it is surjective by the definition of \(d\mathcal{O}_{n,1}\) and thus as \(d\mathcal{B}_{n, 1}\) is a group,  \(d\mathcal{O}_{n,1}\) is also. The identity of \(d\mathcal{O}_{n,1}\) is the image of the identity map, and is thus the single state ``identity" transducer. From Remark~\ref{nV_placement}, we get that \(dV_n\) is the kernel.
\end{proof}
We have now introduced the monoids we need. We now begin showing that \(d\mathcal{B}_{n, 1}\) coincides with the normalizer of \(dV_n\) in \(H(\mathfrak{C}_n^d)\) (the case with \(d=1\) was done in \cite{bleak2016}).
\begin{lemma}\label{main_lemma}
Let \(h\in N_{H(\mathfrak{C}_n^d)}(dV_n)\) and \(s,t\in (X_n^*)^d\backslash \{1_{(X_n^*)^d}\}\). Let \(q_h:= (1_{(X_n^*)^d}){q_{T_h}}_Q\). There exists \(K_{h, s,t}\in \mathbb{N}\) such that for all \(a\in (X_n^*)^d\) with \(\min(\makeset{|a\boldsymbol{\pi}_i|}{\(i\in \overline{d}\)})\geq K_{h,s,t}\), we have \((q_h,sa)\pi_{M_h} = (q_h,ta)\pi_{M_h}\).
\end{lemma}
\begin{proof}
For all \(x\in (X_n^*)^d\), let \(q_x := (q_h, x)\pi_{M_h}\). Let \(f\in dV_n\) be such that \(f\) replaces the prefix \(s\) with the prefix \(t\). By the choice of \(h\), there is some \(g\in nV\) such that \(h^{-1}fh=g\) and so \(fh=hg\).

Let \(q_f:= (1_{(X_n^*)^d}){q_{T_f}}_Q\) and \(q_g:= (1_{(X_n^*)^d}){q_{T_g}}_Q\). It follows that \(f_{M_fM_h, (q_f,q_h)} =fh=hg= f_{M_hM_g, (q_h, q_g)}.\) Let \(I_f, I_g\) be the core states of \(M_f\) and \(M_g\) respectively (which don't do anything). Note that
\[((q_f, q_h), s)\pi_{M_fM_h} = (I_f, q_t),\quad ((q_h, q_g), s)\pi_{M_hM_g} = (q_s, (q_g,(q_h, s)\lambda_{M_h})\pi_{M_g}).\]
Let \(K\in \mathbb{N}\) be such that for all \(w\in (X_n^*)^d\) with \(\min(\makeset{|w\boldsymbol{\pi}_i|}{\(i\in \overline{d}\)})\geq K\), we have
\(\min(\makeset{|((q_h,s)\pi_{M_h}, w)\lambda_{M_h}\boldsymbol{\pi}_i|}{\(i\in \overline{d}\)})\) is at least the synchronizing length of \(M_g\).

Let \(a\in (X_n^*)^d\) be arbitrary such that \(\min(\makeset{|w\boldsymbol{\pi}_i|}{\(i\in
\overline{d}\)}\geq K\). We have
\[((q_f, q_h), sa)\pi_{M_fM_h} = (I_f, q_{ta}),\quad ((q_h, q_g), sa)\pi_{M_hM_g} = (q_{sa}, I_g).\]
Thus, for all \(v\in (X_n^\omega)^d\) we have
\begin{align*}
((q_f, q_h),sa)\lambda_{M_fM_h}(v)f_{M_h, q_{ta}}
&=((q_f, q_h),sa)\lambda_{M_fM_h}(v)f_{M_fM_h, (I_f,q_{ta})}\\
&=(sav)fh\\
&= (sav)hg\\
&=((q_h, q_g),sa)\lambda_{M_hM_g}(v)f_{M_hM_g, (q_{sa},I_g)}\\
&=((q_h, q_g),sa)\lambda_{M_hM_g}(v)f_{M_h, q_{sa}}
\end{align*}

It follows that \(((q_f, q_h),sa)\lambda_{M_fM_h}\) and \(((q_h, q_g),sa)\lambda_{M_hM_g}\) are comparable in each coordinate.

If there was a coordinate in which \(((q_f, q_h),sa)\lambda_{M_fM_h}\) and \(((q_h, q_g),sa)\lambda_{M_hM_g}\) differed, it would follow that either the map \(f_{M_h, q_{ta}}\) or \(f_{M_h, q_{sa}}\) has its image contained in a proper cone. This is impossible as \(M_h\) by definition has no incomplete response. Thus \(((q_f, q_h),sa)\lambda_{M_fM_h} = ((q_h, q_g),sa)\lambda_{M_hM_g}\).

From the equality
\[((q_f, q_h),sa)\lambda_{M_fM_h}(v)f_{M_h, q_{ta}}
=((q_h, q_g),sa)\lambda_{M_hM_g}(v)f_{M_h, q_{sa}}\]
it follows that \(f_{M_h, q_{ta}}=f_{M_h, q_{sa}}\). As \(M_h\) is minimal it follows that \(q_{ta}=q_{sa}\) as required.
\end{proof}
\begin{lemma}\label{hard_aut_containment}
The group \(N_{H(\mathfrak{C}_n^d)}(dV_n)\) is contained in \(d\mathcal{B}_{n, 1}\).
\end{lemma}
\begin{proof}
The proof of this is essentially the same as Corollary 6.17 of \cite{bleak2016}. The idea is as follows. We need only show containment in \(d\mathcal{S}_{n, 1}\) because \(N_{H(\mathfrak{C}_n^d)}(dV_n)\) is a group with the same identity as \(d\mathcal{S}_{n, 1}\). Thus we need only show the synchronizing condition. So it suffices to show that, for all \(h\in N_{H(\mathfrak{C}_n^d)}(dV_n)\), there is a \(K\in \N\) such that the state reached by reading an arbitrary word from \((1_{(X_n^*)^d}){q_{T_h}}_Q\) is determined by the last \(K\) letters of the word (in every coordinate). We do this by collapsing an arbitrary given input word from the front by repeated applications of Lemma~\ref{main_lemma} using \(s\) with size \(1\) and \(t\) with size \(2\) (where size means the sum of the lengths of the coordinates).
\end{proof}
We now recall the theorem of Rubin which connects our arguments to automorphism groups:

\begin{theorem}[Rubin's Theorem \cite{Rubin}]
Let \(G\) be a group of homeomorphisms of a perfect, locally compact, Hausdorff topological space \(X\). For \(U\subseteq X\) let \(G_U := \{g\in G: (x)g=x \text{ for all }x\in X\backslash U\}\). Suppose further that for all \(x\in X\) and \(U\) a neighbourhood of \(x\), we have \((x)G_U\) is somewhere dense. If \(\phi: G \to G\) is a group isomorphism then there is a \(\psi_\phi\in H(X)\) such that \((g)\phi=\phi_{\phi}^{-1}g\psi_{\phi}\) for all \(g\in G\).
\end{theorem}
In \cite{bleak2016}, it is shown that Rubin's theorem allows us to naturally embed \(\A(G_{n, r})\) into \(H(\mathfrak{C}_{n, r})\). This same argument also applies to \(dV_n\), and in fact to any group with an action satisfying the hypothesis of Rubin's theorem.
\begin{corollary}\label{rubin_cor}
The groups \(\A(dV_n)\) and \(N_{H(\mathfrak{C}_n^d)}(dV_n)\) are isomorphic.
\end{corollary}
\begin{theorem}\label{autnV}
The groups \(\A(dV_n)\) and \(d\mathcal{B}_{n,1}\) are isomorphic.
\end{theorem}
\begin{proof}
We have \(\A(dV_n)\cong N_{H(\mathfrak{C}_n^d)}(dV_n)\) by Corollary~\ref{rubin_cor}, and we have 
\[N_{H(\mathfrak{C}_n^d)}(dV_n)\subseteq d\mathcal{B}_{n, 1}\subseteq N_{H(\mathfrak{C}_n^d)}(dV_n)\]
by Lemma~\ref{hard_aut_containment} and Lemma~\ref{nV_is_normal}.
\end{proof}
\begin{corollary}\label{outrubin}
The groups \(\Out(dV_n)\) and \(d\mathcal{O}_{n,1}\) are isomorphic.
\end{corollary}
\begin{proof}
This follows from Theorem~\ref{autnV} and Lemma~\ref{nV_is_normal}.
\end{proof}
\begin{corollary}
For \(d,m\in \N\backslash \{0\}\) and \(n\in \mathbb{N}\backslash \{0,1\}\) the group \(\A(dV_n)^m\) embeds in the group \(\A((md)V_n)\).
\end{corollary}
\begin{proof}
If \((d\mathcal{B}_n)^m\) acts on \((\mathfrak{C}_n^d)^m\cong\mathfrak{C}_n^{dm}\) in the natural fashion, then these homeomorphisms are contained in the group \((md)\mathcal{B}_n\).
\end{proof}
\begin{corollary}
The group \(\A(dV_n)\) is countably infinite.
\end{corollary}
\begin{proof}
The group \(dV_n\) is countably infinite and the group \(d\mathcal{O}_{n,1}\) is countable (it consists of the isomorphism classes of finite things).
\end{proof}
\section{A closer look at the groups \(d\mathcal{O}_{n,1}\)}
We now want to pin down what the core transducers representing \(d\mathcal{O}_{n,1}\) look like. We're going to end up with a semidirect product, so we'll deal with the acting part of the product first. Before that we introduce a notation which we will use repeatedly throughout this section.
\begin{defn}
We define
\[F_{d, n, S}:=\makeset{w\in (X_n^*)^d}{\( w\boldsymbol{\pi}_i = \varepsilon\) for all \(i\in \overline{d}\backslash S\)}\] 
That is \(F_{d, n, S}\) is the submonoid of \((X_n^*)^d\), consisting of those elements only allowed to be non-trivial in the coordinates in \(S\).
\end{defn}
\begin{lemma}\label{notinjlemma}
Let \(T\) be a transducer representing an element of \(\widetilde{d\mathcal{O}_{n,1}}\). If \(i\in \overline{d}\) then there is a unique \((i)\psi_T \in \overline{d}\), such that for all \(q\in Q_T\) and \(l\in F_{d, n, \{i\}}\), we have \((q, l)\lambda_T\in F_{d, n , \{(i)\psi_T\}}\).
\end{lemma}
\begin{proof}
We start by showing the existence of \((i)\psi_T\). First note that for all \(q\in Q_T\), the map \(f_{T, q}\) is necessarily injective (Remark~\ref{inj_clo}).
Suppose for a contradiction that there is \(i\in \overline{d}\), \(l_0, l_1\in F_{d,n,\{i\}}\), \(q_0,q_1\in Q_T\) and \(\alpha, \beta\in \overline{d}\) such that \(\alpha\neq \beta\), \(|(q_0, l_0)\lambda_T\boldsymbol{\pi}_{\alpha}|>0\) and \(|(q_1, l_1)\lambda_T\boldsymbol{\pi}_{\beta}|>0\). We may assume without loss of generality that \(\alpha = 0\) and \(\beta = 1\).

For all \(j\in \overline{d}\backslash \{0, 1\}\) let \(q_j\in Q_T\), \(i_j\in \overline{d}\) and \(l_j\in F_{d,n,\{i_j\}}\) be such that \(|(q_j, l_j)\lambda_T\boldsymbol{\pi}_j|>0\) (note these must exist as \(T\) is non-degenerate and the \(F\) sets generate \((X_n^*)^d\)). Also let \(i_0=i_1 = i\).

It is now the case that if we read a word in coordinate \(i_j\), it's possible to write in coordinate \(j\) (if we're in the correct state). Moreover \(i_{0}=i_1\), so there is a coordinate \(b\) such that we can write words in any given coordinate without reading from coordinate \(b\). We will use this observation to contradict injectivity.

For each state \(q_j\) we choose some \(w_j\in (q_j)\mathfrak{s}_T^{-1}\). If we read \(w_jl_j\) from anywhere we will write non-trivially into the coordinate \(j\). Consider \(w_j\) as \(w_{j,b}w_{j}'\) where \(w_{j,b}\in F_{d, n, \{b\}}\) and \(w_j'\in F_{d, n, \overline{d}\backslash\{b\}}\).

Consider the elements 
\[s_m:= w_{0, b}w_{1, b}\ldots w_{d-1, b}(w_{0}'l_0w_{1}'l_1\ldots w_{d-1}'l_{d-1})^m\in (X_n^*)^d.\]
As the \(w_{j,b}\) type elements commute will all other kind of elements in the product defining of \(s_m\), by commuting the words so that \(w_{j, b}(w_{0}'l_0w_{1}'l_1\ldots w_{d-1}'l_{d-1})^m\) is a part of the product defining \(s_m\), it follows that for all \(q\in Q_T\) we have
\[|(q, s_m)\lambda_T\pi_j|\geq m\]
for all \(j\).

Thus all elements of \((X_n^\omega)^d\) which have all the \(s_m\) as prefixes have the same image under \(f_{T, q}\). This is a contradiction as there are infinitely many such elements and \(f_{T, q}\) is injective.

It remains to show the uniqueness of \((i)\psi_T\). The only way \((i)\psi_T\) could be non-unique is if \(T\) never writes anything when reading from coordinate \(i\). In this case it follows from the existence of the other \((j)\psi_T\), that there is some coordinate into which \(T\) never writes, which is impossible as \(T\) is non-degenerate.
\end{proof}
\begin{defn}
If \(T\) is a transducer representing an element of \(\widetilde{d\mathcal{O}_{n,1}}\), then we define \(\psi_T:\overline{d} \to \overline{d}\) to be the map which was shown to be well-defined in Lemma~\ref{notinjlemma}.
\end{defn}

\begin{theorem}\label{first_semidirect}
The group \(d\mathcal{O}_{n, 1}\) is isomorphic to \(d\mathcal{K}_{n, 1}\rtimes S_d\), where \(S_d\) acts by permuting the coordinates of \(\mathfrak{C}_n^d\) and \(d\mathcal{K}_{n, 1}= \makeset{[T]_{\cong_S}\in d\mathcal{O}_{n, 1}}{\(\psi_T=id\)}\).
\end{theorem}
\begin{proof}
Note that the map \([T]_{\cong_S} \to \psi_{T}\) is a monoid homomorphism to the full transformation monoid on \(d\) points. As \(d\mathcal{O}_{n, 1}\) is a group, it follows that \(\psi_{T}\) is always a permutation for \([T]_{\cong}\in d\mathcal{O}_{n, 1}\). To see that the map is onto the symmetric group and the extension splits, note that for an arbitrary \(f\in S_d\) the map 
\[(p_0,p_1, \ldots, p_{d-1}) \xrightarrow[]{h} (p_{(0)f},p_{(1)f}, \ldots, p_{(d-1)f})\]
is an element of \(d\mathcal{B}_{n, 1}\). Moreover \([\C(M_h)]_{\cong_S}\) maps to \(f\) under the homomorphism.
\end{proof}
We next need to understand the group \(d\mathcal{K}_{n, 1}\). To this end, we recall the groups \(\mathcal{O}_{n, n-1}\) of \cite{bleak2016}. These are groups of synchronizing core \((1,n)\)-transducers, which are isomorphic to the outer automorphism groups of \(G_{n, n-1}\).

In \cite{bleak2016}, it was shown that \(\mathcal{O}_{n, n-1}\) contains \(\mathcal{O}_{n, j}\) for all \(j\), and that a \((1,n)\)-transducer represents an element of \(\mathcal{O}_{n, n-1}\) if and only if it is minimal (in the sense of GNS), synchronizing, it is its own core, all it's states are injective, all its states have clopen image and it's invertible. In particular \(1\mathcal{O}_{n, 1}=\mathcal{O}_{n, 1}\) is a subgroup of \(\mathcal{O}_{n, n-1}\).
\begin{theorem}\label{product_decomposition}
If \([T]_{\cong_S}\in d\mathcal{K}_{n, 1}\) then there are \(T_0, T_1, \ldots T_{d-1}\in \mathcal{O}_{n, n-1}\) such that
\[T \cong_S \prod_{ i\in \overline{d}}T_{i}.\]
\end{theorem}
\begin{proof}
For each \(i\in \overline{d},\) let
\[\sim_{i}:= \makeset{(p, q)\in Q_T^2}{ there is \(w\in F_{d, n, \{i\}}\) with \((p, w)\pi_T=q\)}.\]
One can check that each \(\sim_{i}\) is an equivalence relation. If \(q\in Q_T\), \(i\in \overline{d}\), we restrict the domain and range of \(T\) to \(F_{d, n, \{i\}}\) and restrict the state set of \(T\) to \([q]_{\sim_i}\), then we obtain a subtransducer \(S_{q, i}\) of \(T\).

Moreover for all \(w\in F_{d, n, \overline{d}\backslash \{i\}}\), one can check that the map \(q\mapsto (q, w)\pi_T\) is a strong transducer isomorphism from \(S_{q, i}\) to \(S_{(q,w)\pi_T, i}\).

We will show that \(S_{q, i}\in \mathcal{O}_{n,n-1}\) (if we make the natural identification between \(F_{d, n, \{i\}}\) and \(X_n^*\)).
It suffices to check that the conditions given in \cite{bleak2016} are satisfied. The transducer \(S_{q, i}\) has no inaccessible states by the definition of \(\sim_i\), it has complete response because \(T\) has complete response, it's synchonizing and it's own core because \(T\) is and it has injective state functions because \(T\) does (Remark~\ref{inj_clo}).
By Remark~\ref{inj_clo}, each state function \(f_{T, q}\) of \(T\) has clopen image.
As the image of a state function \(f_{S_{q, i}, p}\)of \(S_{q, i}\) is a projection of the image of the corresponding state function \(f_{T, p}\) of \(T\) (because \(\psi_T\) is well-defined), it follows that this image of \(f_{T, p}\) is compact and open (hence clopen).
It remains to show that \(S_{q, i}\) has no distinct but equivalent states. As \(T\) has no distinct but equivalent states, it suffices to show that if \(p_0, p_1\) are equivalent states in \(S_{q, i}\), then they are equivalent in \(T\).
Let \(j\in \overline{d}\), and \(w\in F_{d, n, \{j\}}\). It suffices to show that \((p_0, w)\lambda_{T}= (p_1, w)\lambda_{T}\).
If \(j= i\) then this follows by the assumption on \(p_0, p_1\). Otherwise let \(s\in F_{d, n, \{i\}}\) be such that \((p_0, s)\pi_T= p_1\).
Then \[(p_0, w)\lambda_{T}\boldsymbol{\pi}_j=(p_0, ws)\lambda_{T}\boldsymbol{\pi}_j=(p_0, sw)\lambda_{T}\boldsymbol{\pi}_j=((p_0, s)\lambda_{T}(p_1, w)\lambda_{T})\boldsymbol{\pi}_j=(p_1, w)\lambda_{T}\boldsymbol{\pi}_j\]
as required. So we can conclude that \(S_{q, i}\in \mathcal{O}_{n, n-1}\) (if we make the natural identification between \(F_{d, n, \{i\}}\) and \(X_n^*\)).

For each \(i\in \overline{d},\) let \(S_{q, i}\) be isomorphic to \(T_{i}\in \mathcal{O}_{n, n-1}\) via an isomorphism which uses \(\boldsymbol{\pi}_i\) as the domain and range isomorphisms (recall that \(q\) has no affect on the isomorphism type). Moreover let \(\phi_{q, i}:S_{q, i} \to T_i\) be the unique such transducer isomorphism (this is unique as the image of an arbitrary state is determined by any one of its synchronizing words). Let \(\phi_{i}:T\to T_i\) be the transducer homomorphism with
\[{\phi_i}_D={\phi_i}_T=\boldsymbol{\pi}_i,\quad {\phi_{i}}_Q=\bigcup_{q\in Q_T}{\phi_{q,i}}_Q.\]

We then define \(\phi:T \to \prod_{ i\in \overline{d}}T_{i}\) to be the unique transducer homomorphism such that for all \(i\in \overline{d}\) we have
\[\phi\boldsymbol{\pi}_i=\phi_i.\]

We need to show that \(\phi\) is a strong isomorphism. We have \(\phi_D=\phi_R=id\) by definition, so we need only show that \(\phi_Q\) is a bijection. It must be surjective as it's target is synchronizing, and so each state is the image of the state reached in \(T\) by reading one of it's synchronizing words. Injectivity is also immediate as \(T\) was assumed to be minimal and hence has no proper strong quotients.
\end{proof}
It is routine to check that the multiplication in \(d\mathcal{K}_{n, 1}\) is also compatible with the multiplication in \(\mathcal{O}_{n, n-1}\), so we can make the following definition:
\begin{defn}\label{alpha}
We define an embedding \(\alpha:d\mathcal{K}_{n, 1}\to \mathcal{O}_{n,n-1}^d\) by
\[\prod_{i\in \overline{d}}([T]_{\cong_S} )\alpha\boldsymbol{\pi}_i \cong_{S } T.\]
\end{defn}
\begin{corollary}\label{wreath_embedding}
The group \(\Out(dV_n)\) embeds in the group
\(\mathcal{O}_{n, n-1}\mathlarger{\mathlarger{\mathlarger{\wr}}} S_d.\)
\end{corollary}
\begin{proof}
This follows from Definition~\ref{alpha}(Theorem~\ref{product_decomposition}), Theorem~\ref{first_semidirect} and Corollary~\ref{outrubin}.
\end{proof}
We now have a connection between \(d\mathcal{O}_{n, 1}\) and \(\mathcal{O}_{n, n-1}\). To pin this down precisely we recall the map \(\overline{\text{sig}}:\mathcal{O}_{n, n-1} \to (\mathbb{Z}/(n-1)\mathbb{Z}, \times)\) of \cite{olukoya2019automorphisms} Definition 7.6.

This group homomorphism takes an \(\mathcal{O}_{n, n-1}\) transducer to the unique element 
\((m + (n-1)\mathbb{Z})~\in~\mathbb{Z}/(n-1)\mathbb{Z}\), such that when a cone in read through the transducer, the transducer writes \(m\) disjoint cones. We can naturally use this map to define a new homomorphism from \(\mathcal{O}_{n, n - 1}^d \).
\begin{defn}
We define the homomorphism \(\overline{\text{sig}_d}:\mathcal{O}_{n, n - 1}^d \to (\mathbb{Z}/(n-1)\mathbb{Z}, \times)\) by
\[(\boldsymbol{T})\overline{\text{sig}_d}=\prod_{i\in \overline{d}} ((\boldsymbol{T})\boldsymbol{\pi}_i)\overline{\text{sig}}.\]
\end{defn}
Moreover we observe that this definition functions as one might expect:
\begin{lemma}
If \(\boldsymbol{T}\in \mathcal{O}_{n, n - 1}^d \), then \((\boldsymbol{T})\overline{\text{sig}_d}\) the unique element \((m + (n-1)\mathbb{Z})\) of \(\mathbb{Z}/(n-1)\mathbb{Z}\), such that \(m\) disjoint cones are written when a cone in read through \(\prod_{i\in \overline{d}}(\boldsymbol{T}\boldsymbol{\pi}_i)\).
\end{lemma}
\begin{proof}
This is well-defined by Remark~\ref{anti-chain sizes}. The result follows from the observations that a \((d, n)\) cone is the same as a product of \(d\) \((1, n)\) cones, and the set of words writable from a state in a product of transducers is the product of the sets of words which can be written in each coordinate.
\end{proof}
In Proposition 7.7 in \cite{olukoya2019automorphisms} and Theorem 9.5 of \cite{bleak2016}, the signature map has been used to identify the groups \(\mathcal{O}_{n, r}\) inside of \(\mathcal{O}_{n, n-1}\). We can now do the same with \(d\mathcal{K}_{n, 1}\) via a similar argument.
\begin{lemma}\label{lemma2}
If \(\boldsymbol{T} \in (\mathcal{O}_{n, n-1})^d\), then
\(P:=\prod_{i\in \overline{d}} \boldsymbol{T}\boldsymbol{\pi}_i\)
is strongly isomorphic to a transducer representing an element of \(d\mathcal{K}_{n, 1}\) if and only if \(\boldsymbol{T} \in \operatorname{ker}(\overline{\text{sig}_d})\).
\end{lemma}

\begin{proof}
\((\Rightarrow):\) Suppose that \(f~\in~d\mathcal{B}_{n, 1}\) is such that the transducer \(P\) is strongly isomorphic to \(\C(M_f)\). If \(U\) is clopen in \(\mathfrak{C}_n^d\), then let \(\operatorname{count}(U)\) be the smallest number of cones in a decomposition of \(U\) into cones. Let \(k\) be the synchronizing length of \(M_f\). It follows that
\begin{align*}
    1+(n-1)\mathbb{Z}    &= \operatorname{count}((\mathfrak{C}_n^d)f) + (n-1)\mathbb{Z}\\
    &=\sum_{w\in (X_n^k)^d} \operatorname{count}((\mathfrak{C}_n^d)f_{M_f,(w) \mathfrak{s}_{M_f}}) + (n-1)\mathbb{Z}\\
     &=\sum_{w\in (X_n^k)^d} (\boldsymbol{T})\overline{\text{sig}_d}  + (n-1)\mathbb{Z}\\
     &=n^{kd} (\boldsymbol{T})\overline{\text{sig}_d}  + (n-1)\mathbb{Z}\\
    &=(\boldsymbol{T})\overline{\text{sig}_d}  + (n-1)\mathbb{Z}.
\end{align*}
The result follows.

\((\Leftarrow):\) Let \(q\in Q_P\) be arbitrary. Then let \(\bigcup_{a\in A} a\mathfrak{C}_n^d\) be a decomposition of \(\im(f_{P, q})\) into disjoint cones.
We have that \(|A| \in 1 + (n-1)\mathbb{Z}\).
Let \(k\in \mathbb{N}\) be greater that \(|A|\), and such that for all \(w\in (X_n^k)^d\) and \(i\in \overline{d}\), we have \(|(q, w)\lambda_P\boldsymbol{\pi}_i| \geq \max\{|a\boldsymbol{\pi}_j|:a\in A, j\in \overline{d }\}\).
For all \(w\in (X_n^k)^d\) let \(a_w\in A\) be such that \(a_w\) is a prefix of \((q, w)\lambda_P\). For all \(a\in A\) we now have that
\[\{((q, w)\lambda_P - a_w)\im(f_{P, (q, w)\pi_T}):w\in (X_n^k)^d \text{ has } a_w = a\}\]
is a partition of \(\mathfrak{C}_n^d\).

As \(|A|\in 1 + (n-1)\mathbb{N}\), there is a complete prefix code \(B\) of size \(|A|\). Let \(\phi: A \to B\) be a bijection. It follows that
\[\{(a_w)\phi((q, w)\lambda_P - a_w)\im(f_{P, (q, w)\pi_P}):w\in (X_n^k)^d\}\]
 is a partition of \(\mathfrak{C}_n^d\).
 We now define an element \(f\in H(\mathfrak{C}_n^d)\) as follows:
 
 If \(w\in (X_n^k)^d\) and \(\overline{x}\in \mathfrak{C}_n^d \) then
 \[(w\overline{x})f = (a_w)\phi((q, w)\lambda_P - a_w)(\overline{x})f_{P, (q, w)\pi_P}.\]
It is routine to verify that \(M_f\) is synchronizing and has core strongly isomorphic to \(P\). Thus \(f\in d\mathcal{S}_{n,1}\), and \(P\) represents an element of \(\widetilde{d\mathcal{O}_{n,1}}\). As \(\operatorname{ker}(\overline{\text{sig}_d})\) is a group, \(\boldsymbol{T}^{-1}\in \operatorname{ker}(\overline{\text{sig}_d})\) and so by the same argument \(P':=\prod_{i\in \overline{d}} (\boldsymbol{T}\boldsymbol{\pi}_i)^{-1}\) also represents and element of \(\widetilde{d\mathcal{O}_{n,1}}\). Thus \(P\) is \(d\mathcal{O}_{n,1}\). As \(P\) is a product of \((1, n)\)-transducers we also have \(\psi_P=id\), so the result follows.
\end{proof}
We now have all the tools to prove Theorem~\ref{2ndtheorem} from the introduction, (the statement is a bit simpler now that we've defined \(\overline{\text{sig}_d}\)).
\begin{theorem}
For all \(d \geq 1\) and \(n\geq 2\) we have
\(\Out(dV_n) \cong \operatorname{ker}(\overline{\text{sig}_d})\rtimes  S_d,\)
where the action of \(S_d\) is the standard permutation of coordinates.
\end{theorem}
\begin{proof}
By Corollary~\ref{outrubin} we have \(\Out(dV_n) \cong d\mathcal{O}_{n,1}\). Thus the result follows from Theorem~\ref{first_semidirect} and Lemma~\ref{lemma2}.
\end{proof}
\begin{corollary}
For all \(d\geq 1\) we have
\(\Out(dV) \cong \Out(V)\mathlarger{\mathlarger{\mathlarger{\wr}}} S_d\) (using the standard action of \(S_d\) on \(d\) points).
\end{corollary}
\begin{proof}
This follows from the previous theorem together with the observation that \((\mathbb{Z}/(2-1)\mathbb{Z}, \times)\) is the trivial group.
\end{proof}
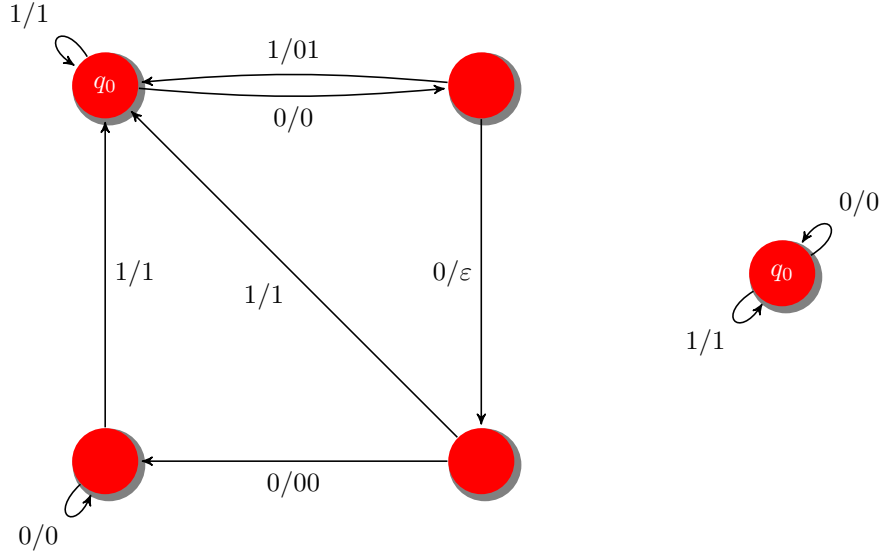
\begin{figure}
\begin{center}
\begin{tikzpicture}[->,>=stealth',shorten >=1pt,auto,node distance=5cm,on grid,semithick,
                    every state/.style={fill=red,draw=none,circular drop shadow,text=white}]  \node [state] (A)                {$q_0$};
  \node [state] (B)  [right= of A] {};
  \node [state] (C)  [below= of B] {};
  \node [state] (D)  [below= of A] {};
  \node [at={(9,-2.5)},state] (E) {$q_0$};
 \path [->]
 (A) edge [out=122,in=147, loop] node [swap]{$1/1$} (A)
 
 (A) edge [out=355,in=185]        node [swap]{$0/0$} (B)
 
 (B) edge [out=175,in=5]        node [swap]{$1/01$} (A)
 (B) edge                        node [swap]{$0/\varepsilon$} (C)
 (C) edge                        node {$1/1$} (A)
 (C) edge                        node {$0/00$} (D)
 (D) edge                        node [swap]{$1/1$} (A)
 (D) edge [out=223,in=247, loop] node [swap]{$0/0$} (D)
 (E) edge [out=212,in=237, loop] node [swap]{$1/1$} (E)
 (E) edge [out=32,in=57, loop] node [swap]{$0/0$} (E);
 \end{tikzpicture}
 \end{center}
\caption{Two transducers with domain and range \(X_2^*\)}\label{figure1}
\end{figure}

\section{Rationality and Representations}\label{ref}
Unfortunately, unlike with \(\mathcal{B}_{2, 1}\), representing elements of \(2\mathcal{B}_{2, 1}\) with transducers can sometimes result in a transducer which has infinitely many states. In Figure~\ref{baker} we see that the baker's map when represented by a transducer, in the way described in this paper, has infinitely many states.

However, if we want our pictures to be finite, then we can consider the submonoid of \((X_n^*)^d\) consisting of those elements \(w\in (X_n^*)^d\) such that \(|w\boldsymbol{\pi}_i|\) is the same for all \(i\).
As elements of \(d\mathcal{B}_{n,1}\) are synchronizing, it follows that if we restrict the domain of a minimal transducer representing an element of \(d\mathcal{B}_{n, 1}\) to this submonoid, then we will only need finitely many states to represent it.
In Figure \ref{figure4} we see the baker's map represented in this fashion.
The main problem with this representation is that composing functions represented with these transducers is much harder (due to the fact that the transducers have distinct domains and ranges).

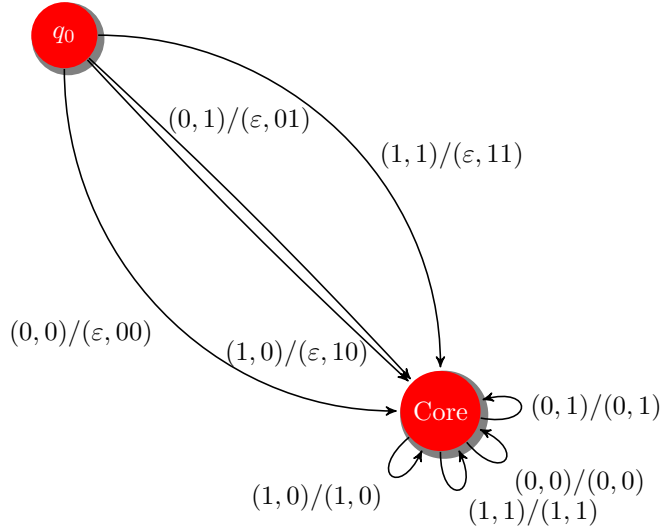
\begin{figure}
\begin{center}
\begin{tikzpicture}[->,>=stealth',shorten >=1pt,auto,node distance=5cm,on grid,semithick,
                    every state/.style={fill=red,draw=none,circular drop shadow,text=white}]  
  \node [state] (A)                {$q_0$};
  \node [] (B)  [right= of A]              {};
  \node [state] (C)  [below= of B] {Core};
 \path [->]
 (A) edge [out=313,in= 137] node [swap, xshift=50, yshift = -40]
 {$(1, 0)/(\varepsilon, 10)$} (C)
 (A) edge [out=317,in=133] node [swap, yshift = 45, xshift = 25]
 {$(0, 1)/(\varepsilon, 01)$} (C)
 (A) edge [out=0,in=90] node [swap, xshift=73]
 {$(1,1)/(\varepsilon,11)$} (C)
 (A) edge [out=270,in=180] node [swap]
 {$(0,0)/(\varepsilon,00)$} (C)

 (C) edge  [out=310,in=335, loop]
 node  [swap]{$(0, 0)/(0, 0)$} (C)
 (C) edge  [out=350,in=15, loop]
 node [swap]{$(0, 1)/(0, 1)$} (C)
 (C) edge [out=220,in=245, loop] node [swap]
 {$(1,0)/(1, 0)$} (C)
 (C) edge [out=270,in=295, loop] node [swap]
 {$(1,1)/(1, 1)$} (C);
 \end{tikzpicture}
 \end{center}
\caption{The subtransducer of the transducer in Figure \ref{baker} with \(\{w\in (\{0, 1\}^*)^2: |w\boldsymbol{\pi}_0|=|w\boldsymbol{\pi}_1|\}\) as domain and no longer accessible states removed.}\label{figure4}
\end{figure}
Figure~\ref{figure1} displays two elements of \(\mathcal{B}_{2,1}\), where the \(q_0\) are the initial states.
These transducers are particularly nice as they are finite. The first one acts by swapping the strings \(``0"\) and \(``00"\) wherever it sees them, and the second one is the identity. These transducers are special in that they are each equal to their own cores.
Thus, these pictures also represent elements of \(\mathcal{O}_{2,1}\) if we ignore the choice of initial state.
In Figure~\ref{figure2} we see the categorical product of the transducers in Figure~\ref{figure1}, which represents an element of \(2\mathcal{B}_{2,1}\) (and also an element of \(2\mathcal{O}_{2,1}\)).

\begin{figure}
\begin{center}
\begin{tikzpicture}[->,>=stealth',shorten >=1pt,auto,node distance=5cm,on grid,semithick,
                    every state/.style={fill=red,draw=none,circular drop shadow,text=white}]  \node [state] (A)                {$q_0$};
  \node [state] (B)  [right= of A] {};
  \node [state] (C)  [below= of B] {};
  \node [state] (D)  [below= of A] {};
 \path [->]
 (A) edge [out=122,in=147, loop] node [swap]{$(1, \varepsilon)/(1, \varepsilon)$} (A)
 (A) edge [out=355,in=185]        node [swap]{$(0, \varepsilon)/(0, \varepsilon)$} (B)
 (A) edge [out=182,in=207, loop] node [swap]{$(\varepsilon,1)/( \varepsilon,1)$} (A)
 (A) edge [out=62,in=87, loop] node [swap]{$(\varepsilon,0)/( \varepsilon,0)$} (A)
 
 (B) edge [out=175,in=5]        node [swap]{$(1, \varepsilon)/(01, \varepsilon)$} (A)
 (B) edge                        node [swap]{$(0, \varepsilon)/(\varepsilon, \varepsilon)$} (C)
 (B) edge [out=10,in=345, loop] node [swap, xshift = 55pt]{$(\varepsilon,1)/( \varepsilon,1)$} (B)
 (B) edge [out=75,in=100, loop] node [swap]{$(\varepsilon,0)/( \varepsilon,0)$} (B)
 
 (C) edge                        node {$(1, \varepsilon)/(1, \varepsilon)$} (A)
 (C) edge                        node {$(0, \varepsilon)/(00, \varepsilon)$} (D)
 (C) edge [out=260,in=285, loop] node [swap]{$(\varepsilon,1)/( \varepsilon,1)$} (C)
 (C) edge [out=340,in=5, loop] node [swap]{$(\varepsilon,0)/( \varepsilon,0)$} (C)
 
 (D) edge                        node [swap,xshift = -60pt]{$(1, \varepsilon)/(1, \varepsilon)$} (A)
 (D) edge [out=223,in=247, loop] node [swap]{$(0, \varepsilon)/(0, \varepsilon)$} (D)
 (D) edge [out=282,in=307, loop] node [swap]{$(\varepsilon,1)/( \varepsilon,1)$} (D)
 (D) edge [out=142,in=167, loop] node [swap]{$(\varepsilon,0)/( \varepsilon,0)$} (D);
 \end{tikzpicture}
 \end{center}
\caption{The transducer obtained by taking the categorical product of the transducers in Figure~\ref{figure1}}\label{figure2}
\end{figure}
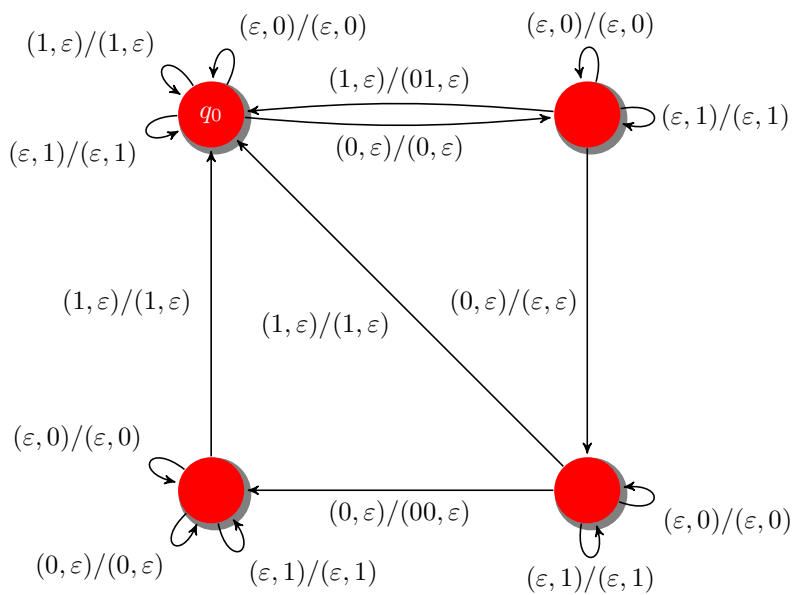

In \cite{bleak2016}, it is shown that \(\mathcal{B}_{2,1}\) consists of rational homeomorphims, and in \cite{Belk2016} Theorem 5.2 it is shown that \(2V\) naturally embeds in \(\mathcal{R}_4\) via conjugation by a homeomorphism between \(\mathfrak{C}_4\) and \(\mathfrak{C}_2^2\). The map of \cite{Belk2016} Theorem 5.2 acts by converting a pair of words \((x_0x_1 \ldots, y_0y_1 \ldots)\) to a word \((x_0,y_0)(x_1, y_1) \ldots\) over the alphabet \(\{(0, 0), (0, 1), (1, 0), (1, 1)\}\). It is natural to ask if the same map gives an embedding of \(2\mathcal{B}_{2,1}\) into \(\mathcal{R}\). However it is routine to verify that if we conjugate the homeomorphism defined by the transducer of Figure~\ref{figure2}, then the resulting map is not rational. This happens because this transducer is a product of a transducer which never resizes words, with a transducer which does this arbitrarily amounts (for example by reading the words \((01)^n\)). So it seems that there is no good way to describe the groups \(\AdVn\) using only finite transducers.

\bibliography{bib.bib}
\bibliographystyle{plain}
\end{document}